\documentclass{amsart}
\usepackage{latexsym, amsmath, amsfonts, amssymb, geometry, mathrsfs, fancyhdr, tikz, tikz-cd}
\usetikzlibrary{matrix,arrows,decorations.pathmorphing,calc}

\def\ds{\displaystyle}

\def\n{\noindent}

\def\E{\mathbb{E}}

\def\v{\varphi}
\def\e{\varepsilon}

\def\O{\Omega}
\def\e{\epsilon}
\def\d{\delta}

\def\P{\mathcal{P}}

\def\X{\mathcal{X}}

\def\dx{d_{\X}}

\newtheorem{theorem}{Theorem}[section]

\theoremstyle{definition}

\theoremstyle{remark}

\numberwithin{equation}{section}

\begin{document}

\title{Intrinsic Isometric Embeddings of Pro-Euclidean Spaces}

%    Information for first author
\author{B. Minemyer}
%    Address of record for the research reported here
\address{Division of Mathematics, Alfred University, Alfred, New York 14802}
%    Current address
%\curraddr{Department of Mathematical Sciences,
%Binghamton University, Binghamton, New York 13902}
\email{minemyer@alfred.edu}
%    \thanks will become a 1st page footnote.
%\thanks{I was supported by the Mathematical Sciences Department at Binghamton University}

%    Information for second author
%\author{Author Two}
%\address{Mathematical Research Section, School of Mathematical Sciences,
%Australian National University, Canberra ACT 2601, Australia}
%\email{two@maths.univ.edu.au}

%    General info
%\subjclass{Primary 54C40, 14E20; Secondary 46E25, 20C20}

\date{November 27, 2013.}

%\dedicatory{This paper is dedicated to our authors.}


\begin{abstract}

In \cite{Petrunin1} Petrunin proves that a metric space $\X$ admits an intrinsic isometry into $\mathbb{E}^n$ if and only if $\X$ is a pro-Euclidean space of rank at most $n$.  He then shows that either case implies that $\X$ has covering dimension $\leq \, n$.  In this paper we extend this result to include embeddings.  Namely, we first prove that any pro-Euclidean space of rank at most $n$ admits an intrinsic isometric embedding into $\mathbb{E}^{2n+1}$.  We then discuss how Petrunin's result implies a partial converse to this result.

\end{abstract}

\maketitle

%\section*{This is an unnumbered first-level section head}
%This is an example of an unnumbered first-level heading.

%\specialsection*{Introduction}
%This is an example of a special section head.

%beginning of section 1: Introduction
\section{Introduction}\label{Introduction}

In \cite{Petrunin1} Petrunin proves the following Theorem:

\begin{theorem}[Petrunin]\label{Petrunin}

A compact metric space $\X$ admits an intrinsic\footnote{For the definition of an intrinsic isometry, please see \cite{Petrunin1}} isometry into $\mathbb{E}^n$ if and only if $\X$ is a pro-Euclidean space of rank at most $n$.  Either of these statements implies that dim($\X$) $\leq$ $n$ where \text{dim}($\X$) denotes the covering dimension of $\X$.

\end{theorem}

A metric space $\X$ is called a \emph{pro-Euclidean space of rank at most $n$} if it can be represented as an inverse limit of a sequence of $n$-dimensional Euclidean polyhedra $\{ \P_i \}$.  In this definition, the inverse limit is taken in the category with objects metric spaces and morphisms short\footnote{1-Lipschitz} maps.  So the limit $\ds{\lim_{\longleftarrow}\P_i = \X}$ means that the sequence $\{ \P_i \}$ converges to $\X$ in both the topological and metric sense.  The morphisms being short maps means that \emph{all} maps involved in the inverse system are short, including the projection maps.

In \cite{Minemyer2} the following Theorem is proved:

\begin{theorem}[M]\label{Minemyer}

Let $\P$ be an $n$-dimensional Euclidean polyhedron, let $f:\P \rightarrow \E^N$ be a short map, and let $\e>0$ be arbitrary.  Then there exists an intrinsic isometric embedding $h: \P \rightarrow \mathbb{E}^N$ which is an $\e$-approximation of $f$, meaning $ |f(x) - h(x)| < \e $ for all $x \in \X$, provided $N \geq 2n + 1$.

\end{theorem}

Combining Theorem \ref{Minemyer} with some methods used by Nash in \cite{Nash1} and Petrunin in \cite{Petrunin1} we can prove the following:

\begin{theorem}\label{Main Theorem}

Let $\X$ be a compact pro-Euclidean space of rank at most $n$.  Then $\X$ admits an intrinsic isometric embedding into $\mathbb{E}^{2n+1}$.

\end{theorem}

Theorem \ref{Main Theorem} extends half of Theorem \ref{Petrunin} to the case of intrinsic isometric embeddings.  What Theorem \ref{Main Theorem} does \emph{not} prove is that, if $\X$ admits an intrinsic isometric embedding into $\mathbb{E}^{2n+1}$, then $\X$ is a pro-Euclidean space of rank at most $n$.  The (main) reason that Theorem \ref{Main Theorem} does not say this is because it is not true!  If $\X$ is a pro-Euclidean space with rank at most $n$, then dim($\X$) $\leq \, n$.  But there are many metric spaces with covering dimension greater than $n$ that admit intrinsic isometric embeddings into $\mathbb{E}^{2n+1}$ (a simple example is the $2n$-sphere).

A metric space $\X$ is a \emph{pro-Euclidean space of finite rank} if $\X$ can be written as an inverse limit of a sequence of Euclidean polyhedra $\{ \P_i \}$ and if there exists a natural number $N$ such that dim($\P_i$) $\leq \, N$ for all $i$.  Again we require that the inverse limit take place in the category of metric spaces with short maps.  Then what we can say by using Theorems \ref{Petrunin} and \ref{Main Theorem} is the following:

\begin{theorem}\label{Main Theorem 2}

A compact metric space $\X$ admits an intrinsic isometric embedding into $\mathbb{E}^N$ for some $N$ if and only if $\X$ is a pro-Euclidean space of finite rank.

\end{theorem}

\subsection*{Acknowledgements}  This paper combines the results of \cite{Petrunin1} and \cite{Minemyer2}.  Both of these papers make great use of a result due to A. Akopyan.  Akopyan's original proof of his result can be found in \cite{Akopyan}.  An English version of his proof can be found in \cite{Minemyer3}.  During the preparation of this paper the author received tremendous support from both Binghamton University and Alfred University.  This research was partially supported by the NSF grant of Tom Farrell and Pedro Ontaneda, DMS-1103335.

%beginning of section 2: Preliminaries
\section{Proof of Theorem \ref{Main Theorem}}

\begin{proof}[Proof of Theorem \ref{Main Theorem}]
	
	Let $\X$ be a pro-Euclidean space of rank at most $n$ and let $(\P_i, \v_{j, i})$ be the inverse system associated to $\X$ where $\P_i$ is an $n$-dimensional Euclidean polyhedron for all $i$.  For each $i$ let $\psi_i: \X \rightarrow \P_i$ be the projection map.  Remember that every map associated with this system is short.
	
	Given $\e_{i + 1} > 0$ and a pl intrinsic isometric embedding $f_i:\P_i \rightarrow \mathbb{E}^{2n + 1}$, by Theorem \ref{Minemyer} there exists a pl intrinsic isometric embedding $f_{i + 1}: \P_{i + 1} \rightarrow \mathbb{E}^{2n + 1}$ such that

$$ |f_{i + 1}(x) - (f_i \circ \v_{i + 1, i})(x)|_{\mathbb{E}^{2n + 1}} < \e_{i + 1} $$

\n for all $x \in \P_{i + 1}$.

Then define $h_i:= f_i \circ \psi_i$ for all $i$. To keep track of all of the maps involved, please see Figure \ref{tenthfig}.  What needs to be shown is that the values for $\e_i$ can be chosen in such a way that the sequence $\{ h_i \}_{i = 0}^{\infty}$ converges uniformally to an intrinsic isometric embedding.  The fact that the sequence $\{ \e_i \}_{i = 1}^{\infty}$ can be chosen so that the sequence $\{ h_i \}_{i = 0}^{\infty}$ converges uniformally to an intrinsic isometry is identical to the proof by Petrunin in \cite{Petrunin1} and is omitted here.

To see that we can choose the sequence $\{ \e_i \}_{i = 1}^{\infty}$ so that the sequence $\{ h_i \}_{i = 0}^{\infty}$ converges to an embedding just consider the collection of sets

$$ \O_i := \{ (x, x') \in \X \times \X \, | \, \dx (x, x') \geq 2^{-i} \} . $$

Since $\ds{\X = \lim_{\longleftarrow}\P_i}$ and because $\O_i$ is compact, for every $i \in \mathbb{N}$ there exists $i'$ such that $\psi_{i'}(x) \neq \psi_{i'}(x')$ for all $(x, x') \in \O_i$.  For every $i$ choose $i' > (i - 1)'$ which satisfies the above.  Thus $h_{i'}(x) \neq h_{i'}(x')$ for all $(x, x') \in \O_i$.  Then we let 

$$\d_i := \text{inf} \{ |h_{i'}(x) - h_{i'}(x')|_{\mathbb{E}^{2n + 1}} \, | \, (x, x') \in \O_i \} > 0. $$

If we choose $\e_i < \frac{1}{4} \text{min} \{ \d_i , \e_{i-1} \} $ then no point pair in $\O_i$ can come together in the limit.  Eventually any pair of distinct points is contained in some $\O_i$, which completes the proof.

\begin{figure}
\begin{center}
\begin{tikzpicture}

\draw (0,2.5)node{$\X$};
\draw (0,0)node{$\P_i$};
\draw (-2.5,0)node{$\P_{i-1}$};
\draw (2.5,0)node{$\P_{i+1}$};
\draw (0,-2.5)node{$\mathbb{E}^{2n+1}$};

\draw (-0.3,0) -- (-2.1,0);
\draw (-2.1,0) -- (-2,0.1);
\draw (-2.1,0) -- (-2,-0.1);

\draw (0.3,0) -- (2.1,0);
\draw (0.3,0) -- (0.4,0.1);
\draw (0.3,0) -- (0.4,-0.1);

\draw (-4.7,0) -- (-2.9,0);
\draw (-4.7,0) -- (-4.6,0.1);
\draw (-4.7,0) -- (-4.6,-0.1);

\draw (4.7,0) -- (2.9,0);
\draw (2.9,0) -- (3,0.1);
\draw (2.9,0) -- (3,-0.1);

\draw (0,2.2) -- (0,0.3);
\draw (0.1,0.4) -- (0,0.3);
\draw (-0.1,0.4) -- (0,0.3);

\draw (0,-2.2) -- (0,-0.3);
\draw (0.1,-2.1) -- (0,-2.2);
\draw (-0.1,-2.1) -- (0,-2.2);

\draw (2.25,0.25) -- (0.25,2.25);
\draw (2.25,0.25) -- (2.05,0.25);
\draw (2.25,0.25) -- (2.25,0.45);

\draw (-2.25,0.25) -- (-0.25,2.25);
\draw (-2.25,0.25) -- (-2.05,0.25);
\draw (-2.25,0.25) -- (-2.25,0.45);

\draw (2.25,-0.25) -- (0.25,-2.25);
\draw (0.25,-2.25) -- (0.25,-2.05);
\draw (0.25,-2.25) -- (0.45,-2.25);

\draw (-2.25,-0.25) -- (-0.25,-2.25);
\draw (-0.25,-2.25) -- (-0.25,-2.05);
\draw (-0.25,-2.25) -- (-0.45,-2.25);

\draw[fill=black!] (-4.9,0) circle (0.1ex);
\draw[fill=black!] (-5.1,0) circle (0.1ex);
\draw[fill=black!] (-5.3,0) circle (0.1ex);
\draw[fill=black!] (4.9,0) circle (0.1ex);
\draw[fill=black!] (5.1,0) circle (0.1ex);
\draw[fill=black!] (5.3,0) circle (0.1ex);

\draw (0,1.25)node[right]{$\psi_i$};
\draw (1.25,1.25)node[right]{$\psi_{i+1}$};
\draw (-1.25,1.25)node[left]{$\psi_{i-1}$};
\draw (-1.15,0)node[above]{$\v_{i, i-1}$};
\draw (1.15,0)node[above]{$\v_{i+1, i}$};
\draw (0,-1.25)node[right]{$f_i$};
\draw (1.25,-1.25)node[right]{$f_{i+1}$};
\draw (-1.25,-1.25)node[left]{$f_{i-1}$};

\draw (3,-2)node[right]{$h_i := f_i \circ \psi_i$ for all $i$};

\end{tikzpicture}
\end{center}
\caption{Diagram for the proof of Theorem \ref{Main Theorem}.}
\label{tenthfig}
\end{figure}
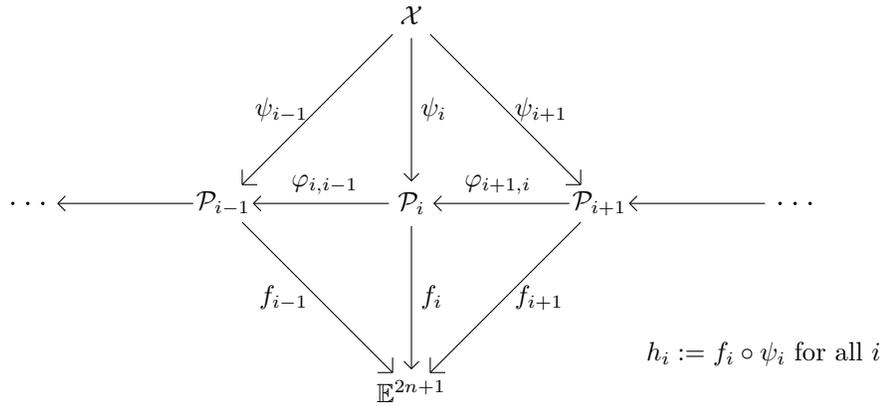

	\end{proof}

\bibliographystyle{amsplain}

\end{document}